\documentclass[12pt]{amsart} %amsartにしたい
\usepackage{amsthm, amssymb, color}
\usepackage{amsmath}
\usepackage{txfonts}
\usepackage{epsfig,multicol}

\usepackage{eufrak}

\usepackage{pict2e}

\theoremstyle{plain} \numberwithin{equation}{section}
\newtheorem{theo}{Theorem}[section]

\newtheorem{lemm}[theo]{Lemma}

\theoremstyle{definition}
\newtheorem*{defi}{Definition}
\newtheorem*{rema}{Remark}
\newtheorem*{exam}{Example}
\newtheorem*{claim}{Claim}

\def\mG{\mathcal G}

\def\mHT{H_T}
\def\Z{\mathbb Z}

\def\Q{\mathbb Q}
\def\l{\ell}
\def\t{t}
\def\w{w}
\def\s{\sigma}
\def\mL{\mathcal L}
\def\mA{\mathcal A}
\def\mB{\mathcal B}
\def\mC{\mathcal C}
\def\mD{\mathcal D}
\def\mT{\frak A_n}

\DeclareMathOperator{\Hom}{Hom}
\DeclareMathOperator{\Map}{Map}
\DeclareMathOperator{\rank}{rank}


\begin{document}
\title[Cohomology ring of the GKM graph]{The cohomology ring of the GKM graph\\ of a flag manifold of classical type}
\author[Y. Fukukawa]{Yukiko Fukukawa}  \author[H. Ishida]{Hiroaki Ishida} \author[M. Masuda]{Mikiya Masuda}
\address{Department of Mathematics, Osaka City University, Sumiyoshi-ku, Osaka 558-8585, Japan.}
\email{yukiko.fukukawa@gmail.com} \email{masuda@sci.osaka-cu.ac.jp}
\address{Osaka City University Advanced Mathematical Institute, Sumiyoshi-ku, Osaka 558-8585, Japan.}
\email{ishida@sci.osaka-cu.ac.jp} 
\date{\today}

\begin{abstract}
If a closed smooth manifold $M$ with an action of a torus $T$ satisfies certain conditions, then a labeled graph $\mG_M$ with labeling in $H^2(BT)$ is associated with $M$, which encodes a lot of geometrical information on $M$.  For instance, the \lq\lq graph cohomology" ring $\mHT^*(\mG_M)$ of $\mG_M$ is defined to be a subring of $\bigoplus_{v\in V(\mG_M)}H^*(BT)$, where $V(\mG_M)$ is the set of vertices of $\mG_M$,  and is known to be often isomorphic to the equivariant cohomology $H^*_T(M)$ of $M$. In this paper, we determine the ring structure of $\mHT^*(\mG_M)$ {with $\Z$ (resp. $\Z[\frac{1}{2}]$) coefficients when $M$ is a flag manifold of type A, B or D (resp. C) in an elementary way}.    
\end{abstract} 

\maketitle %タイトル表示

\section{Introduction}\label{sect1}

Let $T$ be a {compact} torus of dimension $n$ and $M$ a closed smooth $T$-manifold.  The equivariant cohomology of $M$ is defined to be the ordinary cohomology of the Borel construction of $M$, that is, 
\[
H^*_T(M):=H^*(ET\times_T M)
\]
where $ET$ denotes the total space of the universal principal $T$-bundle $ET\to BT$ and $ET\times_T M$ denotes the orbit space of $ET\times M$ by the diagonal $T$-action.  
{Throughout this paper, all cohomology groups are taken with $\Z$ coefficients unless otherwise stated}. 
The equivariant cohomology of $M$ contains a lot of geometrical information on $M$.  Moreover it is often easier to compute $H^*_T(M)$ than $H^*(M)$ by virtue of the Localization Theorem which implies that the restriction map 
\begin{equation} \label{iota}
\iota^*\colon H^*_T(M)\to H^*_T(M^T)
\end{equation} 
to the $T$-fixed point set $M^T$ is often injective, in fact, this is the case when $H^{odd}(M)=0$.  When $M^T$ is isolated, $H^*_T(M^T)=\bigoplus_{p\in M^T}H^*_T(p)$ and hence $H^*_T(M^T)$ is a direct sum of copies of a polynomial ring in $n$ variables because $H^*_T(p)=H^*(BT)$.  

Therefore we {suppose that} $H^{odd}(M)=0$ and $M^T$ is isolated. Goresky-Kottwitz-MacPherson \cite{go-ko-ma98} (see also \cite[Chapter 11]{gu-st99}) found that under the further condition that the weights at a tangential $T$-module are pairwise linearly independent at each $p\in M^T$, the image of $\iota^*$ in \eqref{iota} above is determined  by the fixed point sets of codimension one subtori of $T$ when {considering cohomology with $\Q$ coefficients}.  Their result motivated Guillemin-Zara \cite{gu-za01} to associate a labeled graph $\mG_M$ with $M$ and define the \lq\lq graph cohomology" ring $\mHT^*(\mG_M)$ of $\mG_M$, which is a subring of $\bigoplus_{p\in M^T}H^*(BT)$. Then the result of Goresky-Kottwitz-MacPherson can be stated that $H^*_T(M)\otimes \Q$ is isomorphic to $\mHT^*(\mG_M)\otimes\Q$ as graded rings when $M$ satisfies the conditions mentioned above. 

The result of Goresky-Kottwitz-MacPherson can be applied to many important $T$-manifolds $M$ such as flag manifolds, compact smooth toric varieties and so on.  When $M$ is such a nice manifold, $H^*_T(M)$ is known to be often isomorphic to $\mHT^*(\mG_M)$ without tensoring with $\Q$ (see \cite{ha-he-ho05}, \cite{ma-ma-pa07} for example).  In this paper, we determine the ring structure of $\mHT^*(\mG_M)$ {(resp. $\mHT^*(\mG_M)\otimes \Z[\frac{1}{2}]$) in an elementary way when $M$ is a flag manifold of type A, B or D (resp. C)}. 

The equivariant cohomology ring $H^*_T(M)$ of a flag manifold $M$ of classical type is determined (see \cite{fu-pr98} for example) and our computation of $H^*_T(\mG_M)$ confirms that 
%{the restriction map \eqref{iota} gives an isomorphism from $H^*_T(M)$ to $\mHT^*(\mG_M)$ when $M$ is a flag manifold of type $A$, $B$ or $D$.  When $M$ is a flag manifold of type $C$, the restriction map \eqref{iota} does not give an isomorphism from $H^*_T(M)$ to $\mHT^*(\mG_M)$ in general but becomes an isomorphism when we tensor with $\Z[\frac{1}{2}]$.}
{(resp. $H^*_T(M)\otimes\Z[\frac{1}{2}]$)} is isomorphic to $\mHT^*(\mG_M)$ {(resp. $\mHT^*(\mG_M)\otimes \Z[\frac{1}{2}]$) when $M$ is of type A, B or D (resp. C)}.  The main point in our computation is to show that $\mHT^*(\mG_M)$ is generated by some elements which have a simple combinatorial description.  When $M$ is a flag manifold of type $A_{n-1}$, those elements $\tau_1,\dots,\tau_n$ in $\mHT^*(\mG_M)$ correspond to the equivariant first Chern classes in $H^*_T(M)$ of complex line bundles over $M$ obtained from the flags.  One can show that those first Chern classes generate $H^*_T(M)$ over $H^*(BT)$ using topological techniques.  However, our concern is {to compute the graph cohomology $\mHT^*(\mG_M)$ directly, and so} we show that $\tau_1,\dots,\tau_n$ generate $\mHT^*(\mG_M)$ over $H^*(BT)$ in a purely combinatorial or elementary way.

This paper is organized as follows.  In Section~\ref{sect2} we introduce the notion of a labeled graph and its graph cohomology following the notion of GKM graph and its graph cohomology.  We treat type A in Section~\ref{sect3}, which is a prototype of our argument. Type C is treated in Section~\ref{sectC} and the argument is almost the same as type A {if we work over $\Z[\frac{1}{2}]$ coefficients}.  Types B and D can {also} be treated similarly but more subtle arguments are necessary when we work over $\Z$ coefficients.  This is done in Sections~\ref{sectB} and~\ref{sectD}.

This paper is the detailed and improved version of the announcement \cite{fuku10}. 
Recently the first author (\cite{fuku12}) has determined the ring structure of $H^*_T(\mG_M)$ along the line developed in this paper when $M$ is the flag manifold of type $G_2$.

\section{Labeled graphs and graph cohomology} \label{sect2}

Let $T$ be a compact torus of dimension $n$.  Any homomorphism $f$ from $T$ to a circle group $S^1$ induces a homomorphism $f^*\colon H^*(BS^1)\to H^*(BT)$, so assigning $f$ to $f^*(u)$, where $u$ is a fixed generator of $H^2(BS^1)$, defines a homomorphism from $\Hom(T,S^1)$ (the group of homomorphisms from $T$ to $S^1$) to $H^2(BT)$.  As is well-known, this homomorphism is an isomorphism so that we make the following identification 
\[
\Hom(T,S^1)=H^2(BT)
\]
and use $H^2(BT)$ instead of $\Hom(T,S^1)$ throughout this paper. 

Let $\mG$ be a graph with labeling 
\[
\l(e)\in H^2(BT) \quad\text{for each edge $e$ of $\mG$.}
\]  
We call $\mG$ a \emph{labeled graph} in this paper. Remember that $H^*(BT)$ is a polynomial ring over $\Z$ generated by elements in $H^2(BT)$. 

\begin{defi}
The graph cohomology ring of $\mG$, denoted $\mHT^*(\mG)$, is defined to be the subring of $\Map(V(\mG),H^*(BT))=\bigoplus_{v \in V(\mG)} H^*(BT)$, where $V(\mG)$ denotes the set of vertices of $\mG$, satisfying the following condition:
\begin{quote}
$h\in \Map(V(\mG),H^*(BT))$ is an element of $\mHT^{*}(\mG)$ if and only if $h(v)-h(v')$ is divisible by $\l(e)$ in $H^*(BT)$ whenever the vertices $v$ and $v'$ are connected by an edge $e$ in $\mG$.  
\end{quote}
Note that $\mHT^*(\mG)$ has a grading induced from the grading of $H^*(BT)$.  
\end{defi}

\begin{rema}
Guillemin-Zara \cite{gu-za01} introduced the notion of GKM graph motivated by the result of Goresky-Kottwitz-MacPherson \cite{go-ko-ma98}.  It is a labeled graph but requires more conditions on the labeling $\l$ and encodes more geometrical information on a $T$-manifold $M$ when it is associated with $M$.  However, what we are concerned with in our paper is the graph cohomology of $\mG$ defined above and for that purpose we do not need to require any condition on the labeling $\l$ although the labeled graphs treated in this paper are all GKM graphs.  
\end{rema}

Here is an example of a labeled graph arising from a root system, which is our main concern in this paper. 

\begin{exam}
For a root system $\Phi$ in $H^2(BT)$ (with an inner product) we define a labeled graph $\mG_\Phi$ as follows.  The vertex set $V(\mG_\Phi)$ of $\mG_\Phi$ is the Weyl group $W_\Phi$ of $\Phi$, which is generated by reflections $\s_\alpha$ determined by $\alpha\in \Phi$. Two vertices $\w$ and $\w'$ are connected by an edge, denoted $e_{\w,\w'}$, if and only if there is an element $\alpha$ of $\Phi$ such that $\w'=\w\s_\alpha $, and we label the edge $e_{\w,\w'}$ with $w\alpha$.  Since $\s_{\alpha}=\s_{-\alpha}$, this labeling has ambiguity of sign but the graph cohomology ring $\mHT^*(\mG_\Phi)$ is independent of the sign. 

If $G$ is a compact semisimple Lie group with $\Phi$ as the root system and $T$ is a maximal torus of $G$, then the labeled (or GKM) graph associated with $G/T$ is $\mG_\Phi$, see \cite[Theorem 2.4]{gu-ho-za06}.  
\end{exam}

\section{Type $A_{n-1}$} \label{sect3}

Let $\{\t_i\}_{i=1}^n$ be a basis of $H^2(BT)$, so that $H^*(BT)$ can be identified with the polynomial ring $\Z[\t_1,\t_2,\dots,\t_n]$. We choose an inner product on $H^2(BT)$ such that the basis $\{\t_i\}_{i=1}^n$ is orthonormal.  Then 
\begin{equation} \label{root}
\Phi(A_{n-1}):=\{\pm(\t_i-\t_j)  \mid 1\leq i < j \leq n\}
\end{equation}
is a root system of type $A_{n-1}$.  We denote by $\mA_n$ the labeled graph associated with $\Phi(A_{n-1})$.  The graph $\mA_n$ has the permutation group $S_n$ on $n$ letters $[n]=\{1,2,\dots,n \}$ as the vertex set. We use the one-line notation $\w=\w(1)\w(2)\dots\w(n)$ for permutations.  Two vertices $\w, \w'$ are connected by an edge $e_{\w,\w'}$ if and only if there is a transposition $(i,j)\in S_n$ such that $\w'=\w\cdot(i,j)$, in other words, 
\[
\text{$\w'(i)=\w(j)$,\quad $\w'(j)=\w(i)$\quad and\quad $\w'(r)=\w(r)$\ \  for $r\neq i,j$,}
\]
and the edge $e_{\w,\w'}$ is labeled by $\t_{\w(i)}-\t_{\w'(i)}$. 

For each $i=1,\dots,n$, we define elements $\tau_i, \t_i$ of $\Map(V(\mA_n),H^*(BT))$ by 
\begin{equation}  \label{tau}
\tau_i(\w):=\t_{\w(i)},\quad \t_i(\w):=\t_i\quad\text{for $\w\in S_n$}.
\end{equation}
In fact, both $\tau_i$ and $\t_i$ are elements of $\mHT^2(\mA_n)$. 

\begin{rema} 
Let $0 \subset E_1 \subset \cdots \subset E_n$ be the tautological flag of bundles over a flag manifold of $A_{n-1}$ type. They admit natural $T$-actions and one can see that $\tau_i$ corresponds to the equivariant first Chern class $c_1^T(E_i/E_{i-1})$ of the equivariant line bundle $E_i/E_{i-1}$.  
\end{rema}

\begin{exam}
The case $n=3$. The root system $\Phi(A_2)$ is $\{ \pm(\t_i-\t_j)|1\leq i < j \leq 3\}$.
The labeled graph $\mA_3$ and $\tau_i$ for $i=1,2,3$ are as follows. 
\\
\ \ \ \ \ \ 
{\tiny
\setlength{\unitlength}{0.3mm}
\begin{picture}(130,50)
%\begin{picture}(100,50)
 \linethickness{1pt}
 \put(5,0){\line(100,173){15}}
 \put(20,25.95){\line(1,0){30}}
 \put(50,25.95){\line(100,-173){15}}
 \put(65,0){\line(-100,-173){15}}
 \put(50,-25.95){\line(-1,0){30}}
 \put(5,0){\line(100,-173){15}}
 \put(5,0){\line(1,0){60}}
 \put(20,25.95){\line(200,-346){30}}
 \put(20,-25.95){\line(200,346){30}}
  \put(-15,0){213}
  \put(2,28){123}
  \put(50,28){132}
  \put(67,0){312}
  \put(50,-38){321}
  \put(4,-38){231}
  \put(30,29){$\scriptstyle\t_2-\t_3$}
  \put(58,14){$\scriptstyle\t_1-\t_3$}
  \put(58,-20){$\scriptstyle\t_1-\t_2$}
  \put(-5,-55){The labeled graph $\mA_3$}
\end{picture}
\ 
\begin{picture}(90,50)
 \linethickness{1pt}
 \put(5,0){\line(100,173){15}}
 \put(20,25.95){\line(1,0){30}}
 \put(50,25.95){\line(100,-173){15}}
 \put(65,0){\line(-100,-173){15}}
 \put(50,-25.95){\line(-1,0){30}}
 \put(5,0){\line(100,-173){15}}
 \put(5,0){\line(1,0){60}}
 \put(20,25.95){\line(200,-346){30}}
 \put(20,-25.95){\line(200,346){30}}
  \put(-10,0){$\t_2$}
  \put(2,28){$\t_1$}
  \put(50,28){$\t_1$}
  \put(67,0){$\t_3$}
  \put(50,-38){$\t_3$}
  \put(15,-38){$\t_2$}
   \put(30,-55){$\tau_1$}
\end{picture}
\ 
\begin{picture}(90,50)
\linethickness{1pt}
 \put(5,0){\line(100,173){15}}
 \put(20,25.95){\line(1,0){30}}
 \put(50,25.95){\line(100,-173){15}}
 \put(65,0){\line(-100,-173){15}}
 \put(50,-25.95){\line(-1,0){30}}
 \put(5,0){\line(100,-173){15}}
 \put(5,0){\line(1,0){60}}
 \put(20,25.95){\line(200,-346){30}}
 \put(20,-25.95){\line(200,346){30}}
  \put(-10,0){$\t_1$}
  \put(2,28){$\t_2$}
  \put(50,28){$\t_3$}
  \put(67,0){$\t_1$}
  \put(50,-38){$\t_2$}
  \put(15,-38){$\t_3$}
   \put(30,-55){$\tau_2$}
\end{picture} 
\ 
\begin{picture}(90,50)
 \linethickness{1pt}
 \put(5,0){\line(100,173){15}}
 \put(20,25.95){\line(1,0){30}}
 \put(50,25.95){\line(100,-173){15}}
 \put(65,0){\line(-100,-173){15}}
 \put(50,-25.95){\line(-1,0){30}}
 \put(5,0){\line(100,-173){15}}
 \put(5,0){\line(1,0){60}}
 \put(20,25.95){\line(200,-346){30}}
 \put(20,-25.95){\line(200,346){30}}
  \put(-10,0){$\t_3$}
  \put(2,28){$\t_3$}
  \put(50,28){$\t_2$}
  \put(67,0){$\t_2$}
  \put(50,-38){$\t_1$}
  \put(15,-38){$\t_1$}
   \put(30,-55){$\tau_3$}
\end{picture}
} 
\\ \\ 
\end{exam}

%The purpose of this section is to prove the following. 

\vspace{1cm}

\begin{theo} \label{Atype}
Let $\mA_n$ be the labeled graph associated with the root system $\Phi(A_{n-1})$ of type $A_{n-1}$ in \eqref{root}. Then 
\begin{equation*}
\mHT^{*}(\mA_n)=\mathbb{Z}[\tau_1,{\scriptstyle\cdots},\tau_n,\t_1,{\scriptstyle\cdots},\t_n]/(e_i(\tau)-e_i(\t)\mid i=1,{\scriptstyle\cdots},n),
\end{equation*}
where $e_i(\tau)$ {\rm (}resp. $e_i(\t)${\rm )} is the $i^{th}$ elementary symmetric polynomial in $\tau_1,{\scriptstyle\cdots},\tau_n$ 
{\rm (}resp. $\t_1,{\scriptscriptstyle\cdots},\t_n${\rm )}.
\end{theo}

The rest of this section is devoted to the proof of Theorem~\ref{Atype}.  We first prove the following. 

\begin{lemm} \label{surj}
$\mHT^{*}(\mA_n)$ is generated by $\tau_1,{\scriptstyle\cdots},\tau_n,\t_1,{\scriptstyle\cdots},\t_n$ as a ring. 
\end{lemm}

\begin{proof}
We shall prove the lemma by induction on $n$. When $n=1$, $\mHT^{*}(\mA_1)$ is generated by $\t_1$ since $\mA_1$ is a point; so the lemma holds. 

Suppose that the lemma holds for $n-1$.  Then it suffices to show that any homogeneous element $h$ of $\mHT^*(\mA_n)$, say of degree $2k$, can be expressed as a polynomial in the $\tau_i$'s and $\t_i$'s.  
For each $i=1,\dots,n$, we set 
$$V_i:=\{ \w \in S_n \mid \w(i)=n\}.$$ 
{The sets $V_i$ give a decomposition of $S_n$ into disjoint subsets}. 
{We consider the full labeled} subgraph $\mL_i$ of $\mA_n$ with $V_i$ as the vertex set, {where the full subgraph means that any edge in $\mA_n$ connecting vertices in $V_i$ lies in $\mL_i$.  
Note that the vertices of $\mL_i$ can naturally be identified with permutations on $\{1,2,\dots,n\}\backslash \{i\}$ and $\mL_i$ is isomorphic to $\mA_{n-1}$ for any $i$}.  
 
Let 
\begin{equation} \label{q}
1\leq q\leq \min\{k+1,n\}
\end{equation}
and assume that 
\begin{equation} \label{Vi}
\text{$h(v)=0$\quad for any {$v \in \bigcup_{i=1}^{q-1}V_i$}}
\end{equation}
{and that $q$ is the minimal integer with the properties \eqref{q} and \eqref{Vi}}. 

Note that a vertex $\w$ in $V_q$ is connected by an edge in $\mA_n$ to a vertex $v$ in $V_i$ for $i\not=q$ if and only if $v=\w\cdot(i,q)$.  In this case $h(\w)-h(v)$ is divisible by $\t_{\w(i)}-\t_{\w(q)}=\t_{\w(i)}-\t_n$ and $h(v)=0$ whenever $i<q$ by \eqref{Vi}, so $h(\w)$ is divisible by $\t_{\w(i)}-\t_n$ for $i<q$. Thus, for each $\w\in V_q$, there is an element $g^q(\w)\in \Z[\t_1,{\scriptstyle\cdots},\t_n]$ such that 
\begin{equation} \label{fw}
h(\w)=(\t_{\w(1)}-\t_n)(\t_{\w(2)}-\t_n) \dots (\t_{\w(q-1)}-\t_n)g^q(\w)
\end{equation}
where $g^q(\w)$ is homogeneous and of degree $2(k+1-q)$ because $h(\w)$ is homogeneous and of degree $2k$.   

One expresses 
\begin{equation} \label{gr}
g^q(\w)=\sum_{r=0}^{k+1-q} g^q_r(\w)\t_n^{r}
\end{equation}
with homogeneous polynomials $g^q_r(w)$ of degree $2(k+1-q-r$) in $\Z[\t_1,{\scriptstyle\cdots},\t_{n-1}]$.  

\begin{claim}
{For each $r$ with $0\le r\le k+1-q$}, there is a polynomial $G^q_r$ in $\tau_i$'s (except $\tau_q$) and $\t_i$'s (except $\t_n$) with integer coefficients such that $G^q_r(\w)=g^q_r(\w)$ for any $\w\in V_q$. 
\end{claim}

\medskip
\noindent
Proof of Claim. 
If the vertex $\w$ in $V_q$ is connected by an edge in $\mA_n$ to a vertex $v$ in $V_q$, then there is an element $(i,j) \in S_n$ such that $v=\w\cdot(i,j)$ where $i$ and $j$ are not equal to $q$. Since $h$ is an element of $\mHT^{*}(\mA_n)$, $h(\w)-h(v)$ has to be divisible by $\t_{\w(i)}-\t_{\w(j)}$, 
{in other words, 
\begin{equation} \label{hwhv1}
h(w)\equiv h(v) \mod {\t_{w(i)}-\t_{w(j)}}.
\end{equation}
}

{
On the other hand, it follows from \eqref{fw} that we have
\begin{equation} \label{hwhv2}
h(w)=g^q(w)\prod_{s=1}^{q-1}(t_{w(s)}-t_n),\quad  h(v)=g^q(v)\prod_{s=1}^{q-1}(t_{v(s)}-t_n).
\end{equation}
Here, since $v=w\cdot (i,j)$, we have $w(i)=v(j)$, $w(j)=v(i)$ and $w(s)=v(s)$ for $s\not=i,j$.  Moreover $w(i)$ and $w(j)$ are not equal to $n$ because $i$ and $j$ are not equal to $q$.  Therefore
\[
\prod_{s=1}^{q-1}(t_{w(s)}-t_n)\equiv \prod_{s=1}^{q-1}(t_{v(s)}-t_n)\nequiv 0 \mod  {\t_{w(i)}-\t_{w(j)}}. 
\]
This together with \eqref{hwhv1} and \eqref{hwhv2} implies that 
\[
g^q(w)\equiv g^q(v) \mod  {\t_{w(i)}-\t_{w(j)}}
\]
and hence
\[
g^q_r(w)\equiv g^q_r(v) \mod  {\t_{w(i)}-\t_{w(j)}}\quad\text{for any $r$}
\]
because $w(i)$ and $w(j)$ are not equal to $n$. Therefore  $g^q_r(\w)-g^q_r(v)$ is divisible by $\t_{\w(i)}-\t_{\w(j)}$ for any $r$.  
This means that $g^q_r$ restricted to $\mL_q$ is an element of $\mHT^{*}(\mL_q)$.  
The vertices of $\mL_q$ can be identified with permutations on $\{1,\dots,n\}\backslash \{q\}$ and hence $\mL_q$ is naturally isomorphic to $\mA_{n-1}$, so the induction assumption on $n$ implies that there is a polynomial $G^q_r$ in $\tau_i$'s (except $\tau_q$) and $\t_i$'s (except $\t_n$) with integer coefficients such that $G^q_r(\w)=g^q_r(\w)$ for any $\w\in V_q=V(\mL_q)$}, proving the claim. 
\medskip

Since $\tau_i(\w)= \t_{\w(i)}$ and $\w(i)=n$ for $\w\in V_i$, we have  
\begin{equation} \label{prod}
\prod_{j=1}^{q-1} (\tau_j-\t_n)(\w) =0 \quad\text{for any {$\w \in \bigcup_{i=1}^{q-1}V_i$}.} 
\end{equation}
Therefore, it follows from \eqref{fw}, \eqref{gr}, the claim above and \eqref{prod} that putting $G^q=\sum_{r=0}^{k+1-q}G^q_r\t_n^r$, we have 
\begin{equation*}
\begin{split}
\big(h-G^q\prod_{j=1}^{q-1} (\tau_j -\t_n)\big)(\w)
=&h(\w)-g^q(\w)\prod_{j=1}^{q-1} (\t_{\w(j)}-\t_n)\\
=&0\qquad\text{for any {$\w \in \bigcup_{i=1}^qV_i$}.}
\end{split}
\end{equation*}
Therefore, subtracting the polynomial $G^q\prod_{j=1}^{q-1} (\tau_j -\t_n)$ from $h$, we may assume that 
\[
\text{$h(v)=0$\quad for any {$v \in \bigcup_{i=1}^qV_i$}. }
\]
The above argument implies that $h$ finally takes zero on all vertices of $\mA_n$ (which means $h=0$) by subtracting polynomials in $\tau_i$'s and $\t_i$'s with integer coefficients, and this completes the induction step. 
\end{proof}

{Let $k$ be a commutative ring. We take $k=\Z$ or $\Z[\frac{1}{2}]$ later.  Remember that the Hilbert series of a graded $k$-algebra $A^*=\bigoplus_{j=0}^\infty A^j$, where $A^j$ is the degree $j$ part of $A^*$ and assumed to be of finite rank over $k$, is a formal power series defined by 
\[
F(A^*,s):=\sum_{j=0}^\infty(\rank_{k} A^j)s^j.
\]
}

\begin{lemm} \label{inje}
$F(\mHT^*(\mA_n),s)={\frac{1}{(1-s^2)^{2n}}} \prod_{i=1}^n(1-s^{2i})$. 
\end{lemm}

\begin{proof}
We first note that $\mHT^{*}(\mA_n)$ is free over $\Z$ because it is a submodule of $\bigoplus_{w\in S_n}H^*(BT)$.  Let $d_n(k):=\rank_\mathbb{Z} \mHT^{2k}(\mA_n)$.  
Then 
\begin{equation} \label{mHmG}
F(\mHT^*(\mA_n),s) = \sum_{k=0}^\infty d_n(k)s^{2k}.
\end{equation}

{For $q$ with $0\le q\le k+1$, we set 
\[
F^{2k}_q=\{h\in H^{2k}_T(\mA_n)\mid h(w)=0\text{ for any $w\in \bigcup_{i=1}^qV_i$}\}.
\] 
Then we have a filtration 
\[
H^{2k}_T(\mA_n)=F_0^{2k}\supset F_1^{2k}\supset \dots \supset F_k^{2k}\supset F_{k+1}^{2k}=0
\]
and since $g^q_r$ in \eqref{gr} belongs to $\mHT^{2(k+1-q-r)}(\mL_q)=\mHT^{2(k+1-q-r)}(\mA_{n-1})$ as shown in the claim and $g^q_r$ can be chosen arbitrarily, we have 
\[
\rank_\Z F_q^{2k}-\rank_\Z F_{q-1}^{2k}=\sum_{r=0}^{k+1-q}d_{n-1}(k+1-q-r)=\sum_{r=0}^{k+1-q}d_{n-1}(r).
\]
Therefore, noting \eqref{q}, we have 
\begin{equation} \label{eq:dnk}
d_n(k)=\sum_{q=1}^{\min\{k+1,n\}}\sum_{r=0}^{k+1-q}d_{n-1}(r). 
\end{equation}
If we set $d_{n-1}(j)=0$ for $j<0$, then an elementary computation shows that \eqref{eq:dnk} reduces to}  
\begin{equation} \label{dnk}
d_n(k)=\begin{cases} \sum_{i=1}^n i\cdot d_{n-1}(k+1-i)\quad &\text{if $k\le n-1$,}\\
\sum_{i=1}^n i\cdot d_{n-1}(k+1-i)+n\sum_{i=n+1}^{k+1}d_{n-1}(k+1-i)\quad &\text{if $k\ge n$.}
\end{cases}
\end{equation}
We shall abbreviate $F(\mHT^*(\mA_n),s)$ as $F_n(s)$. Then, plugging \eqref{dnk} in \eqref{mHmG}, we obtain
\begin{equation*} \label{Fns}
\begin{split}
F_n(s)           =& \sum_{k=0}^{\infty}\Big(d_{n-1}(k)+2d_{n-1}(k-1)+{\scriptstyle\cdots}+nd_{n-1}(k+1-n)\Big)s^{2k}  \\
            &  + n\sum_{k=n}^\infty \Big(d_{n-1}(k-n)+{\scriptstyle\cdots} + d_{n-1}(1)+d_{n-1}(0)\Big) s^{2k}\\
           =& F_{n-1}(s)+2s^2F_{n-1}(s)+{\scriptstyle\cdots} +ns^{2n-2}F_{n-1}(s) \\
            &  +n\Big(d_{n-1}(0)s^{2n}\frac{1}{1-s^2}+d_{n-1}(1)s^{2n+2}\frac{1}{1-s^2}+{\scriptstyle\cdots}\Big) \\
           =& F_{n-1}(s)\Big(1+2s^2+{\scriptstyle\cdots} +ns^{2n-2}\Big)+n{\frac{s^{2n}}{1-s^2}} F_{n-1}(s) \\
%&=& {\frac{F_{n-1}(s)}{1-s}} \{ (1-s)(1+2s+{\scriptstyle\cdots}+ns^{n-1})+ns^s \} \\          
           =& {\frac{1-s^{2n}}{(1-s^2)^2}}F_{n-1}(s).
           \end{split}
\end{equation*}
On the other hand, $F_1(s)=1/(1-s^2)$ since $\mHT^*(\mA_1)=\Z[\t_1]$. Therefore the lemma follows. 
\end{proof}

We abbreviate the polynomial ring $\mathbb{Z}[\tau_1,{\scriptstyle\cdots},\tau_n,\t_1,{\scriptstyle\cdots},\t_n]$ as $\Z[\tau,\t]$.  The canonical map $\Z[\tau,\t]\to H^*_T(\mA_n)$ is a {degree-preserving} homomorphism which is surjective by Lemma~\ref{surj}.  Let $e_i(\tau)$ (resp. $e_i(\t)$) denote the $i^{th}$ elementary symmetric polynomial in $\tau_1,{\scriptstyle\cdots},\tau_n$ (resp. $\t_1,{\scriptstyle\cdots},\t_n$).  It easily follows from \eqref{tau} that $e_i(\tau)=e_i(\t)$ for $i=1,{\scriptstyle\cdots},n$.  Therefore the canonical map above induces a {degree-preserving}  epimorphism 
\begin{equation} \label{cano}
\mT^*:=\Z[\tau,\t]/\big(e_i(\tau)-e_i(\t) \mid i=1,{\scriptstyle\dots},n\big) \to \mHT^*(\mA_n). 
\end{equation}
We note that $\mT^*$ is a $\Z[t]$-module in a natural way. 

\begin{lemm} \label{gene}
$\mT^*$ is generated by $\{\prod_{p=1}^{n-1}\tau_p^{i_p}\mid i_p\le n-p\}$ as a $\Z[t]$-module.
\end{lemm} 

\begin{proof}
Clearly the elements $\prod_{p=1}^{n-1}\tau_p^{i_p}$, with no restriction on exponents $i_p$, generate $\mT^*$ as a $\mathbb{Z}[\t]$-module.
Therefore, it suffices to prove that $\tau_p^{n-p+1}$ can be expressed as a polynomial in $\tau_1,\dots,\tau_p$ and $t_i$'s with the exponent of $\tau_p$ less than or equal to $n-p$.   

Let $h_i(\t)$ (resp. $h_i(\tau)$) be the $i^{th}$ complete symmetric polynomial in $\t_1,{\scriptstyle \cdots},\t_n$ (resp. $\tau_1,{\scriptstyle \cdots} ,\tau_n$) and $h_0(\t)=e_0(\t)=1$.  Since $e_i(\tau)=e_i(\t)$ for any $i$, we have 
\begin{equation*}%\label{et2}
\prod_{i=1}^{n}(1-{\tau_i} x)=\prod_{i=1}^{n}(1-{\t_i} x)
\end{equation*}
where $x$ is an indeterminate.  It follows that  
%\begin{eqnarray}
\begin{equation} 
\begin{split}
\sum_{i\geq 0}h_{i}(\tau_1,{\scriptstyle \cdots},\tau_{p})x^i &= \prod_{i=1}^{p}{\frac{1}{1-\tau_ix}} \label{et1} \\
&= \prod_{i=p+1}^{n}(1-\tau_ix)\prod_{i=1}^{n}{\frac{1}{1-{\t_i}x}} \\ %\nonumber \\
&= \Big(\sum_{i=0}^{n-p}(-1)^ie_i(\tau_{p+1},{\scriptstyle \cdots},\tau_n)x^i\Big)\Big(\sum_{i\geq 0}h_{i}(\t)x^i\Big). \\ %\nonumber 
%\end{eqnarray}
\end{split}
\end{equation}
Comparing coefficients of $x^{n+1-p}$ in \eqref{et1}, we have
\begin{equation}\label{heh}
h_{n+1-p}(\tau_1,{\scriptstyle \cdots},\tau_p)=\sum_{i=0}^{n-p}(-1)^ie_i(\tau_{p+1},{\scriptstyle \cdots},\tau_n)h_{n+1-p-i}(\t)
\end{equation}
while it easily follows from the definition of $h_i$ that  
\begin{equation}\label{comp}
h_{n+1-p}(\tau_1,{\scriptstyle \cdots},\tau_p)=\tau_p^{n+1-p}+\sum_{i=0}^{n-p}\tau_p^i\cdot h_{n+1-p-i}(\tau_1,{\scriptstyle \cdots},\tau_{p-1}).
\end{equation}
By \eqref{heh} and \eqref{comp} we have 
%\begin{eqnarray}\label{tauA}
\begin{equation} \label{tauA}
\begin{split}
\tau_p^{n+1-p} &= -\sum_{i=0}^{n-p}\tau_p^i\cdot h_{n+1-p-i}(\tau_1,{\scriptstyle \cdots},\tau_{p-1}) \\
                 & \   \  \  +\sum_{i=0}^{n-p}(-1)^ie_i(\tau_{p+1},{\scriptstyle \cdots},\tau_n)h_{n+1-p-i}(\t). %\nonumber 
\end{split}
\end{equation}
%\end{eqnarray}

On the other hand, it follows from $e_i(\tau)=e_i(\t)$ that  
\begin{equation*}
\sum_{j=0}^{i}e_j(\tau_1,{\scriptstyle \cdots},\tau_p)e_{i-j}(\tau_{p+1},{\scriptstyle \cdots},\tau_n)= e_i(\t) \text{$\ \ \ \ \ $for any $i$},
\end{equation*}
that is,
\begin{equation*}
e_i(\tau_{p+1},{\scriptstyle \cdots},\tau_n)=e_i(t)-\sum_{j=1}^{i}e_j(\tau_1,{\scriptstyle \cdots},\tau_p)e_{i-j}(\tau_{p+1},{\scriptstyle \cdots},\tau_n) \text{$\ \ \ \ \ $for any $i$}.
\end{equation*}
Thus one obtains
\begin{eqnarray*}
e_1(\tau_{p+1},{\scriptstyle \cdots},\tau_n) &=& e_1(t)-e_1(\tau_1,{\scriptstyle \cdots},\tau_p)\\
e_2(\tau_{p+1},{\scriptstyle \cdots},\tau_n) &=& e_2(t)-e_2(\tau_1,{\scriptstyle \cdots},\tau_p)-e_1(\tau_1,{\scriptstyle \cdots},\tau_p)e_1(\tau_{p+1},{\scriptstyle \cdots},\tau_n) \\
    &=& e_2(t)-e_2(\tau_1,{\scriptstyle \cdots},\tau_p)-e_1(\tau_1,{\scriptstyle \cdots},\tau_p) \big(e_1(t)-e_1(\tau_1,{\scriptstyle \cdots},\tau_p)\big), 
%\\  e_3(\tau_{p+1},{\scriptstyle \cdots},\tau_n) &=& e_3(t)-\sum_{j=1}^{3}e_j(\tau_1,{\scriptstyle \cdots},\tau_p)e_{3-j}(\tau_{p+1},{\scriptstyle \cdots},\tau_n)   \\
%&=& e_3(t)-e_3(\tau_1,{\scriptstyle \cdots},\tau_p) -e_2(\tau_1,{\scriptstyle \cdots},\tau_p) \{ e_1(t)-e_1(\tau_1,{\scriptstyle \cdots},\tau_p) \}   \\ & \  & -e_1(\tau_1,{\scriptstyle \cdots},\tau_p) 
%\{ e_2(t)-e_2(\tau_1,{\scriptstyle \cdots},\tau_p)-e_1(\tau_1,{\scriptstyle \cdots},\tau_p) \{  e_1(t)-e_1(\tau_1,{\scriptstyle \cdots},\tau_p) \} \}
\end{eqnarray*}
and so on.
This shows that $e_i(\tau_{p+1},{\scriptstyle \cdots},\tau_n)$ can be written as a linear combination of  $\prod_{k=1}^{p}\tau_k^{i_k}$, with $i_k \leq i$, over $\mathbb{Z}[\t]$. 
Therefore, it follows from \eqref{tauA} that $\tau_p^{n+1-p}$ is written as a polynomial in $\tau_1,{\scriptstyle \cdots},\tau_p$ and $\t_i$'s with the  exponent of $\tau_p$ less than or equal to $n-p$.
\end{proof}

Now we are in a position to complete the proof of Theorem~\ref{Atype}.  

\begin{proof}[Proof of Theorem~\ref{Atype}]  
If two formal power series $a(s)=\sum_{i=0}^\infty a_is^i$ and $b(s)=\sum_{i=0}^\infty b_is^i$ with real coefficients $a_i$ and $b_i$ satisfy $a_i\le b_i$ for every $i$, then we express this as $a(s)\le b(s)$.  

The Hilbert series of the free $\Z[t]$-module generated by $\prod_{k=1}^{n-1}\tau_k^{i_k}$ is given by $\frac{1}{(1-s^2)^n}s^{2\sum_{k=1}^{n-1}i_k}$, so it follows from Lemma~\ref{gene} that 
\begin{equation*} %\label{ineq1}
F(\mT^*,s)\le \frac{1}{(1-s^2)^n}\sum_{0\le i_k\le n-k}s^{2\sum_{k=1}^{n-1}i_k}
\end{equation*}
and the equality above holds if and only if %$\mT^*$ is free as a $\Z[t]$-module. 
generators $\prod_{p=1}^{n-1} \tau_p^{i_p}$ with $i_p \leq n-p$ are linearly independent over $\mathbb{Z}[t]$.
 Here the right hand side above is equal to 
\begin{eqnarray*}
\frac{1}{(1-s^2)^n} \sum_{0 \leq i_k \leq n-k} \Big(\prod_{k=1}^{n-1} s^{2i_k} \Big) &=&
\frac{1}{(1-s^2)^n} \prod_{k=1}^{n-1} \Big( \sum_{0 \leq i_k \leq n-k} s^{2i_k} \Big) \\
&=&
\frac{1}{(1-s^2)^n}\prod_{q=1}^{n-1}(1+s^2+\dots+s^{2q}) \\
&=& \frac{1}{(1-s^2)^{2n}}\prod_{i=1}^n(1-s^{2i})\\                                
\end{eqnarray*}
which agrees with $F(\mHT^*(\mA_n),s)$ by Lemma~\ref{inje}.  Therefore $F(\mT^*,s)\le F(\mHT^*(\mA_n),s)$.  On the other hand, the surjectivity of the map \eqref{cano} implies the opposite inequality.  Therefore  $F(\mT^*,s)=F(\mHT^*(\mA_n),s)$.  %This means that the inequality in \eqref{ineq1} must be an equality and hence $\mT^*$ is free as a $\Z[t]$-module, in particular, as a $\Z$-module.  
Since the map \eqref{cano} is surjective and $F(\mT^*,s)=F(\mHT^*(\mA_n),s)$, we conclude that the map \eqref{cano} is actually an isomorphism.  This proves Theorem~\ref{Atype}.
\end{proof}

\section{Type $C_n$}  \label{sectC}

The argument developed in Section~\ref{sect3} works for the case of type $C_n$ with a little modification.  In this section we shall state the result and mention necessary changes in the argument.  

The root system $\Phi(C_n)$ of type $C_n$ is given by  
\begin{equation} \label{rootCn}
\Phi(C_{n})=\{\pm(\t_i + \t_j), \ \pm(\t_i - \t_j), \ \pm 2\t_k \mid 1\le i<j\le n,\ 1\le k\le n\}
\end{equation}
and its Weyl group is the signed permutation group on $\pm[n]:=\{\pm 1,\dots,\pm n\}$, which we denote by $\tilde S_n$.  Namely $w\in \tilde S_n$ permutes elements in $\pm [n]$ up to sign.  Again we use the one-line notation $w=w(1)w(2)\dots w(n)$. The number of elements in $\tilde S_n$ is $2^nn!$.  

Let $\mC_n$ be the labeled graph associated with the root system $\Phi(C_n)$.  It has $\tilde S_n$ as vertices and two vertices $w,w'\in \tilde S_n$ are connected by an edge $e_{w,w'}$ if and only if one of the following occurs:
\begin{enumerate}
\item there is a pair $\{i,j\}\subset [n]$ such that 
\[
\text{$(\w'(i),\w'(j))=\pm(\w(j),\w(i))$\quad and\quad $\w'(r)=\w(r)$\ \  for $r\ (\neq i,j)\in [n]$,}
\]
\item there is an $i\in [n]$ such that 
\[
\text{$\w'(i)=-\w(i)$\quad and\quad $\w'(r)=\w(r)$\ \ for $r\ (\neq i)\in [n]$.}
\]
\end{enumerate}
We understand 
\[
\t_{-m}:=-\t_m \quad\text{for a positive integer $m$}.
\]
Then the edge $e_{\w,\w'}$ is labeled by $\t_{\w(i)}-\t_{\w'(i)}$ in case (1) above and by $2\t_{\w(i)}$ in case (2) above, and the elements $\tau_i$ and $\t_i$ for $i=1,\dots,n$ defined by 
\begin{equation} \label{tauC}
\text{$\tau_i(\w):=\t_{\w(i)}$\quad and\quad $\t_i(\w):=\t_i$}
\end{equation}
belong to $\mHT^2(\mC_n)$.    

{If $M_n$ is a flag manifold of type $C_n$, then the restriction map
\[
H^*_T(M_n) \to \bigoplus_{w\in \tilde S_n}H^*(BT)
\]
is injective and the image is known to be described as 
\[
\mathbb{Z}[\tau_1,{\scriptstyle\cdots},\tau_n,\t_1,{\scriptstyle\cdots},\t_n]/
(e_i({\tau}^2)-e_i({\t}^2) \mid i=1,{\scriptstyle\cdots},n),
\]
where $e_i({\tau}^2)$ {\rm (}resp. $e_i({\t}^2)${\rm )} is the $i^{th}$ elementary symmetric polynomial in ${\tau_1}^2,{\scriptstyle \cdots},{\tau_n}^2$ {\rm (}resp. ${\t_1}^2,{\scriptscriptstyle\cdots},{\t_n}^2${\rm )}, see \cite[Chapter 6]{fu-pr98}. 
So, one may expect that $H^*_T(\mC_n)$ is generated by $\tau_1,\dots,\tau_n,\t_1,\dots,\t_n$ as a ring, but this is not true in general as shown in the following example.  This fact was pointed out by T. Ikeda, L. C. Mihalcea and H. Naruse.  
}

{
\begin{exam}
Take $n=2$.  One can check that $h\in \Map(\tilde S_2,H^*(BT))$ defined by 
\[
h(v)=\begin{cases} 0 \ &\text{if $v(1)=2$, $v(2)=2$ or $(v(1),v(2))=(-2,1)$}\\
-2t_2(t_1-t_2)(t_1+t_2)\ &\text{if $(v(1),v(2))=(1,-2)$}\\
2t_2^2(t_1+t_2)\ &\text{if $(v(1),v(2))=(-1,-2)$}\\
2t_1t_2(t_1+t_2)\ &\text{if $(v(1),v(2))=(-2,-1)$}
\end{cases}
\]
is an element of $H^*_T(\mC_2)$, see Figure~\ref{exam1}.
{\centering{
\begin{figure}[h]
\includegraphics[width=100mm]{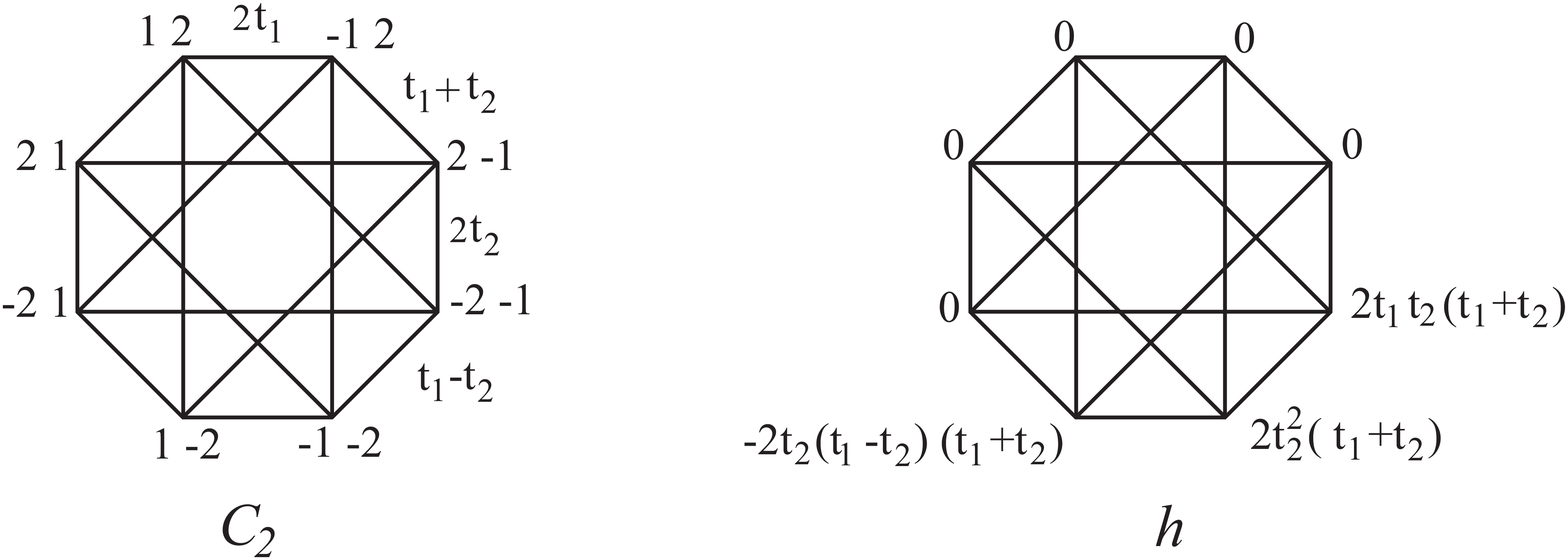}\caption{}\label{exam1}
\end{figure}
}}
In fact, the element $h$ agrees with   
$$\frac{1}{2}(\tau_1-t_2)(\tau_2-t_2)(\tau_1-\tau_2+t_1+t_2)$$
and this shows that $h$ is not a polynomial in $\tau_1,\tau_2,\t_1,\t_2$ over $\Z$. 
\end{exam}
}

{The problem is caused by the presence of the factor $2$ in the root system \eqref{rootCn} and 
if we work over $\Z[\frac{1}{2}]$ instead of $\Z$, then the argument developed in the previous section works with a little modification and we obtain the following.} 

\begin{theo} \label{Ctype}
Let $\mC_n$ be the labeled graph associated with the root system $\Phi(C_{n})$ of type $C_n$ as above. Then 
\begin{equation*} 
\mHT^{*}(\mC_n){\otimes\Z[\frac{1}{2}]}=\mathbb{Z}{[\frac{1}{2}]}[\tau_1,{\scriptstyle\cdots},\tau_n,\t_1,{\scriptstyle\cdots},\t_n]/
(e_i({\tau}^2)-e_i({\t}^2) \mid i=1,{\scriptstyle\cdots},n),
\end{equation*}
where $e_i({\tau}^2)$ {\rm (}resp. $e_i({\t}^2)${\rm )} is the $i^{th}$ elementary symmetric polynomial in ${\tau_1}^2,{\scriptstyle \cdots},{\tau_n}^2$ {\rm (}resp. ${\t_1}^2,{\scriptscriptstyle\cdots},{\t_n}^2${\rm )}.
\end{theo}

The proof of Theorem~\ref{Ctype} is almost same as that of Theorem~\ref{Atype} {and we shall outline it}.  First we prove the following. 

\begin{lemm} \label{surjC}
$\mHT^{*}(\mC_n){\otimes \Z[\frac{1}{2}]}$ is generated by $\tau_1,{\scriptstyle\cdots},\tau_n,\t_1,{\scriptstyle\cdots},\t_n$ as a ring. 
\end{lemm}

\begin{proof}
The proof goes as in Lemma~\ref{surj}. When $n=1$, $\mC_1$ has only one edge with vertices $1$ and $-1$, and the label of the edge is $2\t_1$.  Since $\tau_1(\pm 1)=\pm \t_1$, it is easy to check that the {lemma} holds when $n=1$. 
 
The key step in the proof of Lemma~\ref{surj} was that if $h\in \mHT^*(\mA_n)$ vanishes on $V_i$ for $i<q$, then one could modify $h$ so that it vanishes on $V_i$ for $i<q+1$ by subtracting a {polynomial in $\tau_i$'s and $\t_i$'s with integer coefficients} from $h$, where the polynomial was of the form $G^q\prod_{i=1}^{q-1}(\tau_i-\t_n)$. In the case of type $C_n$, we consider 
$$V_i^\pm:=\{\w\in \tilde S_n \mid \w(i)=\pm n \}$$ 
and the {full labeled} subgraph $\mL_i^\pm$ of $\mC_n$ with $V_i^\pm$ as the vertex set, where $\mL_i^+$ and $\mL_i^-$  are both isomorphic to $\mC_{n-1}$ for each $i=1,\dots,n$. 

The same argument as in the case of type $A_{n-1}$ shows that if $h\in \mHT^*(\mC_n)$ vanishes on $V_i^+$ for $i<q$, then one can modify $h$ so that it vanishes on $V_i^+$ for $i<q+1$ by subtracting from $h$ a  polynomial of the form $G_+^q\prod_{k=1}^{q-1}(\tau_k-\t_n)$ in $\tau_i$'s and $\t_i$'s {with coefficients in $\Z[\frac{1}{2}]$}.  Moreover, if $h$ vanishes on all $V_i^+$  and $V_j^-$ for $j<q$ with some $q\geq 1$, then one can modify $h$ so that it vanishes on all $V_i^+$ and $V_j^-$ for $j<q+1$ by subtracting from $h$ a polynomial in $\tau_i$'s and $\t_i$'s {with coefficients in $\Z[\frac{1}{2}]$} of the form $G_-^q \prod_{k=1}^{n}(\tau_k-\t_n) \prod_{l=1}^{q-1}(\tau_l+\t_n)$.  Therefore we finally reach an element which vanishes on all $V_i^\pm$ by subtracting polynomials in $\tau_i$'s and $\t_i$'s {with coefficients in $\Z[\frac{1}{2}]$} from $h$, and this proves the lemma. 
\end{proof}  

It easily follows from \eqref{tauC} that $e_i(\tau^2)=e_i(\t^2)$ for $i=1,{\scriptstyle\cdots},n$.  Therefore we have a {degree-}preserving epimorphism 
\begin{equation} \label{canoC}
\Z{[\frac{1}{2}]}[\tau,\t]/\big(e_i(\tau^2)-e_i(\t^2) \mid i=1,{\scriptstyle\dots},n\big) \to \mHT^*(\mC_n){\otimes\Z[\frac{1}{2}]}
\end{equation}
and the same argument as in Lemma~\ref{gene} proves the following. 

\begin{lemm} \label{geneC}
The left hand side in \eqref{canoC} is generated by $\prod_{k=1}^{n-1}\tau_k^{i_k}$ with $i_k\le 2(n-k)$ as a $\Z{[\frac{1}{2}]}[t]$-module.
\end{lemm} 

Then, comparing the Hilbert series of the both sides in \eqref{canoC}, we see that the map \eqref{canoC} is an isomorphism.  The details are left to the reader.

\section{Type $B_n$}\label{sectB}

In this section we treat type $B_n$. The root system $\Phi(B_n)$ of type $B_n$ is given by  
\begin{equation} \label{rootBn}
\Phi(B_{n})=\{\pm(\t_i + \t_j), \ \pm(\t_i - \t_j), \ \pm\t_k \mid 1\le i<j\le n,\ 1\le k\le n\}
\end{equation}
and its Weyl group is the same as that of type $C_n$, i.e. the signed permutation group $\tilde S_n$.  

Let $\mB_n$ be the labeled graph associated with the root system $\Phi(B_n)$.  This labeled graph has the same vertices and edges as $\mC_n$.  Their labels are almost same.  The only difference is that the edge $e_{\w,\w'}$ with $\w,\w'$ such that $\w'(i)=-\w(i)$ for some $i\in [n]$ and $\w'(r)=\w(r)$ for $r\ (\neq i)\in [n]$ is labeled by $\t_{\w(i)}$ in $\mB_n$ while it is labeled by $2\t_{\w(i)}$ in $\mC_n$. 

We define $\tau_i$ and $\t_i$ for $i=1,\dots,n$ by \eqref{tauC}.  They belong to $\mHT^2(\mB_n)$.  As remarked above, the only difference between $\mB_n$ and $\mC_n$ is the factor $2$ in the labels on the edges $e_{\w,\w'}$ mentioned above.  Therefore, if we work over $\Z[\frac{1}{2}]$ instead of $\Z$, then the same argument as in the case of type $C_n$ proves the following. 

\begin{lemm} \label{1/2}
\[
\mHT^{*}(\mB_n)\otimes\Z[\frac{1}{2}]=\mathbb{Z}[\frac{1}{2}][\tau_1,{\scriptstyle\cdots},\tau_n,\t_1,{\scriptstyle\cdots},\t_n]/
(e_i({\tau}^2)-e_i({\t}^2) \mid i=1,{\scriptstyle\cdots},n).
\]
\end{lemm}

The above lemma is not true without tensoring with $\Z[\frac{1}{2}]$.  We need to introduce another family of elements to generate $\mHT^*(\mB_n)$ as a ring.  Since $e_i(\tau)(\w)\equiv e_i(\t)(\w) \pmod{2}$ for any $\w$ in $\tilde S_n$, $e_i(\tau)-e_i(\t)$ is divisible by $2$ and one sees that   
\begin{equation*}
f_i :=(e_i(\tau)-e_i(\t))/2
\end{equation*}
is actually an element of $\mHT^{*}(\mB_n)$. Note that $f_0=0$ since $e_0=1$ by definition. The purpose of this section is to prove the following. 

\begin{theo} \label{Btype}
Let $\mB_n$ be the labeled graph associated with the root system $\Phi(B_{n})$ of type $B_n$ in \eqref{rootBn}. Then 
\begin{equation*} %\label{Bnco}
\mHT^{*}(\mB_n)=\mathbb{Z}[\tau_1,{\scriptstyle\cdots},\tau_n,\t_1,{\scriptstyle\cdots},\t_n ,f_1 ,{\scriptstyle\cdots},f_n]/I
\end{equation*}
where $I$ is the ideal generated by 
\[
\begin{split}
& 2f_i-e_i(\tau)+e_i(\t ) \quad (i=1,{\scriptstyle\cdots},n),\\
&\sum_{j=1}^{2k}(-1)^jf_j(f_{2k-j}+e_{2k-j}(\t)) \quad (k=1,{\scriptstyle\cdots},n)
\end{split}
\]
where $f_\ell=e_\ell(t)=0$ for $\ell>n$. 
\end{theo}
\begin{rema}
If we set $t_1=\dots=t_n=0$, then the right hand side of the identity in Theorem~\ref{Btype} reduces to 
\[
\mathbb{Z}[\tau_1,{\scriptstyle\cdots},\tau_n,f_1 ,{\scriptstyle\cdots},f_n]/J
\end{equation*}
where $J$ is the ideal generated by 
\[
 2f_i-e_i(\tau) \quad (i=1,{\scriptstyle\cdots},n),\qquad \sum_{j=1}^{2k-1}(-1)^jf_jf_{2k-j}+f_{2k} \quad (k=1,{\scriptstyle\cdots},n)
\]
where $f_\ell=0$ for $\ell>n$, and this agrees with the ordinary cohomology ring of the flag manifold of type $B_n$, see \cite[Theorem 2.1]{to-wa74}. 
\end{rema}
The idea of the proof of Theorem~\ref{Btype} is same as before but the argument becomes more complicated because of the elements $f_i$'s.  We first observe relations between $f_ i$'s in $\mHT^*(\mB_n)$ and those in $\mHT^*(\mB_{n-1})$.

\begin{lemm} \label{fn}
For $\w$ in $\tilde S_{n}$ with $\w(q)=\pm n$, let ${\w}'$ be an element in $\tilde S_{n-1}$ represented by $\w(1) \cdots \w(q-1) \w(q+1) \cdots \w(n)$.  We denote $f_i$ in $\mHT^*(\mB_n)$ by $f_i^{(n)}$.  Then
\begin{eqnarray}
f_{i}^{(n-1)}({\w}')=\Bigg\{ 
\begin{array}{ll}
\sum_{j=0}^{i-1}f_{i-j}^{(n)}(\w)(-\t_n)^j & \text{if $\w(q) =n $,} \nonumber \\
\sum_{j=0}^{i-1}f_{i-j}^{(n)}(\w)t_n^j+\sum_{j=1}^{i}e_{i-j}(\t_1,{\scriptstyle \cdots},\t_{n-1})\t_n^j & \text{if $\w(q) = -n $}. \nonumber \\
\end{array} 
\end{eqnarray}
\end{lemm}

\begin{proof}
We have
\begin{equation}
e_i(\t_1,{\scriptstyle \cdots},\t_n)-e_i(\t_1,{\scriptstyle \cdots},\t_{n-1}) = e_{i-1}(\t_1,{\scriptstyle \cdots},\t_{n-1}) \t_n \nonumber 
\end{equation}
and 
\begin{equation}
e_i(\tau_1(\w),{\scriptstyle \cdots},\tau_n(\w))-e_i(\tau_1(\w'),{\scriptstyle \cdots},\tau_{n-1}(\w'))
 = e_{i-1}(\tau_1(\w'),{\scriptstyle \cdots},\tau_{n-1}(\w'))\tau_q(\w). \nonumber 
\end{equation} 
Therefore
\begin{eqnarray}
f_{i}^{(n)}(\w )-f_{i}^{(n-1)}({\w}') 
&=&\frac{1}{2} \Big(e_i(\tau_1({\w}),{\scriptstyle \cdots},\tau_n({\w}))-e_i(\t_1,{\scriptstyle \cdots},\t_n)\Big) \nonumber \\
& \ & \ \ \ -\frac{1}{2} \Big(e_i(\tau_1({\w}'),{\scriptstyle \cdots},\tau_{n-1}({\w}'))-e_i(\t_1,{\scriptstyle \cdots},\t_{n-1})\Big) \nonumber \\
&=&\frac{1}{2} \Big(e_{i-1}(\tau_1({\w}'),{\scriptstyle \cdots},\tau_{n-1}({\w}'))\tau_q(\w)- e_{i-1}(\t_1,{\scriptstyle \cdots},\t_{n-1})\t_n \Big) \nonumber \\
&=& \Bigg\{
\begin{array}{ll}
 f_{i-1}^{(n-1)}(\w') \t_n & \text{if $\w(q) =n$,}\nonumber \\
-\big(f_{i-1}^{(n-1)}(\w') +e_{i-1}(\t_1,{\scriptstyle \cdots},\t_{n-1})\big) \t_n & \text{if $\w(q) = -n$}. \nonumber \\
\end{array}
\end{eqnarray}%\hfill $\Box$
Using the above identity repeatedly, we obtain the following for $w$ with $w(q)=n$:
\begin{eqnarray*}
f_{i}^{(n-1)}({\w}') &=& f_{i}^{(n)}(\w )- f_{i-1}^{(n-1)}(\w')\t_n \\
  &=& f_{i}^{(n)}(\w )-\big(f_{i-1}^{(n)}(\w) -f_{i-2}^{(n-1)}(\w')\t_n \big)\t_n  \\
  &=& f_{i}^{(n)}(\w )-f_{i-1}^{(n)}(\w)\t_n  +\big(f_{i-2}^{(n)}(\w)- f_{i-3}^{(n-1)}(\w')\big)\t_n^2  \\
  &  &  \  \  \  \  \  \  \  \  \vdots \\
  &=& \sum_{j=0}^{i-1}f_{i-j}^{(n)}(\w)(-\t_n)^j.
\end{eqnarray*}

The case $\w(q)=-n$ can be treated in the same way.
%It follows that for $V^{\pm}_i \subset \tilde S_{n} \  (i=1,{\scriptstyle \cdots},n)$ there is a polynomial $G$ in $\t_i$'s and $f_i^{(n)}$'s such that $G \big|_{V^{\pm}_i}$ coincides with $f_i^{(n-1)}$. 
\end{proof}

\begin{lemm} \label{surjB}
$\mHT^{*}(\mB_n)$ is generated by $\tau_1,{\scriptstyle\cdots},\tau_n,\t_1,{\scriptstyle\cdots},\t_n,f_1,{\scriptstyle \cdots},f_n$ as a ring. 
\end{lemm}

\begin{proof}
We use induction on $n$ as before. When $n=1$, $\mB_1$ has only one edge with vertices $1$ and $-1$, and the label of the edge is $\t_1$.  Since $\tau_1(\pm 1)=\pm \t_1$, it is easy to check that the lemma holds when $n=1$.  

As before, we consider $V_i^\pm:=\{\w\in \tilde S_n \mid \w(i)=\pm n \}$ and the {full labeled} subgraph $\mL_i^\pm$ of $\mB_n$ with $V_i^\pm$ as the vertex set, where $\mL_i^+$ and $\mL_i^-$  are both isomorphic to $\mB_{n-1}$ for each $i=1,\dots,n$. 
If $h\in \mHT^*(\mB_n)$ vanishes on $V_i^+$ for $i<q$, then one can modify $h$ so that it vanishes on $V_i^+$ for $i<q+1$ by subtracting from $h$ an integer coefficient polynomial of the form $G_+^q\prod_{k=1}^{q-1}(\tau_k-\t_n)$ in $\tau_i$'s, $\t_i$'s and $f_i$'s.  In fact, we obtain $G_+^q$ as an element of $\Map(\tilde S_n,H^*(BT))$ whose restriction to $\mL_q^+$ belongs to $\mHT^*(\mL_q^+)$.  Since $\mL_q^+$ is isomorphic to $\mB_{n-1}$ and $\mHT^*(\mB_{n-1})$ is generated by $\tau_i$'s, $\t_i$'s and $f_i$'s by the induction assumption, we can take $G_+^q$ as a polynomial in $\tau_i$'s, $\t_i$'s and $f_i$'s {with integer coefficients}, where we use  Lemma~\ref{fn}.  

If $h$ vanishes on all $V_i^+$  and $V_j^-$ for $j<q$ with some $q\geq 1$, then one can also modify $h$ so that it vanishes on all $V_i^+$ and $V_j^-$ for $j<q+1$ by subtracting from $h$ some polynomial in $\tau_i$'s, $\t_i$'s and $f_i$'s {with integer coefficients}.  However, this polynomial is not of the form $G_-^q \prod_{k=1}^{n}(\tau_k-\t_n) \prod_{l=1}^{q-1}(\tau_l+\t_n)$ because $\prod_{k=1}^{n}(\tau_k-\t_n)(\w)$ is divisible by $2$ for $\w\in V_i^{-}$. 
Instead of $\prod_{k=1}^{n}(\tau_k-\t_n)$, we use the following element
\begin{eqnarray} \label{elementB}
{\frac{1}{2}}\prod_{k=1}^{n}(\tau_k-\t_n)&=&{\frac{1}{2}}\sum_{k=0}^{n}(-1)^{n-k}e_k(\tau){\t_n}^{n-k} \nonumber\\ 
 &=&{\frac{1}{2}}\sum_{k=0}^{n}(-1)^{n-k}(2f_k+e_k(\t)){\t_n}^{n-k}\\ 
&=&\sum_{k=1}^{n}(-1)^{n-k}f_k {\t_n}^{n-k}, \nonumber
\end{eqnarray}
so that the polynomial which we subtract is of the form 
$$G_-^q \left(\Sigma_{k=1}^{n}(-1)^{n-k}f_k {\t_n}^{n-k}\right) \prod_{l=1}^{q-1}(\tau_l+\t_n)$$
where $G_-^q$ is a polynomial in $\tau_i$'s, $\t_i$'s and $f_i$'s {with integer coefficients}. Thus we finally reach an element which vanishes on all $V_i^\pm$ by subtracting polynomials in $\tau_i$'s, $\t_i$'s and $f_i$'s {with integer coefficients} from $h$, and this proves the lemma. 
\end{proof}  

\begin{lemm} \label{lemm:relation}
$\sum_{i=1}^{2k}(-1)^if_i(f_{2k-i}+e_{2k-i}(\t))=0$ for $k=1,{\scriptstyle\cdots},n$.
\end{lemm}
\begin{proof}
Cleaely we have $e_i(\tau^2)=e_i(\t^2)$ for $i=1,2,\dots,n$, namely
\begin{equation}\label{et2}
\prod_{i=1}^{n}(1-{\tau_i}^2 x^2)=\prod_{i=1}^{n}(1-{\t_i}^2 x^2).
\end{equation}
Therefore
\begin{eqnarray}
0&=&\prod_{i=1}^{n}(1-{\tau_i}^2 x^2)-\prod_{i=1}^{n}(1-{\t_i}^2 x^2)\nonumber\\
 &=&\Big(\sum_{i=0}^{n}(-1)^ie_i(\tau)x^i\Big)\Big(\sum_{j=0}^{n}e_j(\tau)x^j\Big)-\Big(\sum_{i=0}^{n}(-1)^ie_i(\t)x^i\Big)\Big(\sum_{j=0}^{n}e_j(\t)x^j\Big)\nonumber \\
 &=&\Big(\sum_{i=0}^{n}(-1)^i(2f_i+e_i(\t))x^i\Big)\Big(\sum_{j=0}^{n}(2f_j+e_j(\t))x^j\Big)-\Big(\sum_{i=0}^{n}(-1)^ie_i(\t)x^i\Big)\Big(\sum_{j=0}^{n}e_j(\t)x^j\Big)\nonumber \\
 &=&4\sum_{i,j=1}^{n}(-1)^if_i f_jx^{i+j}+2\sum_{i,j=0}^{n}(-1)^i\big(f_i e_j(\t)+f_j e_i(\t)\big)x^{i+j}\ \nonumber \\
 &=&4\sum_{k=1}^{n}{\sum_{i=1}^{2k}(-1)^if_i f_{2k-i} x^{2k}}+4\sum_{k=1}^{n}{\sum_{i=1}^{2k}(-1)^if_ie_{2k-i}(\t)x^{2k}} \nonumber
\end{eqnarray}
where we used $f_0=0$. This implies the lemma because the coefficient of $x^{2k}$ must vanish.  
\end{proof}

We abbreviate the polynomial ring $\mathbb{Z}[\tau_1,{\scriptstyle \cdots},\tau_n,\t_1,{\scriptstyle \cdots},\t_n,f_1,{\scriptstyle \cdots},f_n]$ as $\mathbb{Z}[\tau,\t,f]$.  
Since $2f_i=e_i(\tau)-e_i(\t)$ by definition, it follows from Lemma~\ref{lemm:relation} that
the canonical map $\mathbb{Z}[\tau,\t,f] \rightarrow \mHT^*(\mB_n)$ induces a grade preserving map 
\begin{equation}\label{canoB}
\mathbb{Z}[\tau,\t,f]/I \rightarrow \mHT^*(\mB_n),
\end{equation}
where 
$I$ is the ideal in Theorem~\ref{Btype}, 
and it is an epimorphism by Lemma~\ref{surjB}. Since $\mHT^*(\mB_n)$ is a submodule of a direct sum of some $\mathbb{Z}[\t]$'s, $\mHT^*(\mB_n)$ is free over $\mathbb{Z}$. In addition, its Hilbert series is given by ${\frac{1}{(1-s^2)^{2n}}} \prod_{i=1}^n(1-s^{4i})$. This can be shown by a similar computation to the proof of Lemma~\ref{inje}. In order to prove that the epimorphism \eqref{canoB} is actually an isomorphism, it suffices to verify the following Lemmas~\ref{tB} and ~\ref{tB2}.

\begin{lemm} \label{tB}
$\mathbb{Z}[\tau,\t,f]/I$ is free over $\mathbb{Z}$. 
\end{lemm}

\begin{proof}
By Lemma~\ref{1/2} $\mathbb{Z}[\tau,\t,f]/I \otimes \mathbb{Z}[\frac{1}{2}]=\mathbb{Z}[\tau,\t]/I \otimes \mathbb{Z}[\frac{1}{2}]$ is isomorphic to $\mHT^*(\mB_n)\otimes \Z[\frac{1}{2}]$.  Since $\mHT^*(\mB_n)$ is free over $\Z$, this means that  $\mathbb{Z}[\tau,\t,f]/I$ has no odd torsion and hence it suffices to show that $\mathbb{Z}[\tau,\t,f]/I$ has no 2-torsion. If $\mathbb{Z}[\tau,\t,f]/I$ has 2-torsion, then 
$$F(\mathbb{Z}[\tau,\t,f]/I\otimes \mathbb{Z}/2,s)> F(\mHT^*(\mB_n)\otimes \mathbb{Z}/2,s);$$ 
so we will prove that
\begin{equation} \label{ineq} 
F(\mathbb{Z}[\tau,\t,f]/I \otimes \mathbb{Z}/2,s)\leq F(\mHT^*(\mB_n)\otimes \mathbb{Z}/2,s). 
\end{equation}

\begin{claim}
$\mathbb{Z}[\tau,\t,f]/I \otimes \mathbb{Z}/2$ is generated by 
 elements $\prod_{k=1}^{n}\tau_k^{i_k} \prod_{k=1}^{n} f_k^{j_k} $, with $i_k \leq n-k$ and $j_k \leq 1$, over $\mathbb{Z}/2[\t]$. 
\end{claim}

We admit the claim for the moment and complete the proof of the lemma.  If the elements $\prod_{k=1}^{n}\tau_{k}^{i_k}\prod_{k=1}^{n}f_{k}^{j_k}$ are linearly independent over $\Z/2[t]$, then the Hilbert series of $\mathbb{Z}[\tau,\t,f]/I \otimes \mathbb{Z}/2$ (over the field $\Z/2$) is given by $${\frac{1}{(1-s^2)^n}}\sum_{0\leq i_k \leq n-k} \sum_{0 \leq j_k \leq 1}s^{2(\sum_{k=1}^{n}i_k + \sum_{k=1}^{n}kj_k)} ,$$ so we have  
\begin{eqnarray}\label{moutyotto}
F(\mathbb{Z}[\tau,\t,f]/I \otimes \mathbb{Z}/2,s) &\leq & {\frac{1}{(1-s^2)^n}}\sum_{0\leq i_k \leq n-k} \sum_{0 \leq j_k \leq 1}s^{2(\sum_{k=1}^{n}i_k + \sum_{k=1}^{n}kj_k)}  \nonumber \\
&=& {\frac{1}{(1-s^2)^n}} \Big( \sum_{0\leq i_k \leq n-k} \prod _{k=1}^n s^{2i_k} \Big) \Big( \sum_{0 \leq j_k \leq 1} \prod_{k=1}^{n}s^{2kj_k} \Big)  \nonumber \\
&=&{\frac{1}{(1-s^2)^{2n}}}(1-s^2)^n\prod_{i=1}^{n-1}(1+\sum_{j=1}^{i}s^{2j})\prod_{i=1}^{n}(1+s^{2i})  \\
 &=&{\frac{1}{(1-s^2)^{2n}}}\prod_{i=1}^{n}(1-s^{2i})\prod_{i=1}^{n}(1+s^{2i}) \nonumber \\
 &=&{\frac{1}{(1-s^2)^{2n}}}\prod_{i=1}^{n}(1-s^{4i}) \nonumber \\
 &=&F(\mHT^*(\mB_n)\otimes \mathbb{Z}/2,s). \nonumber 
\end{eqnarray}
This proves the desired inequality \eqref{ineq}.

In the sequel it remains to show the claim above and for that it suffices to verify the following (I) and (I\hspace{-0.5mm}I):

\noindent
(I) Elements $\prod_{k=1}^{n}\tau_k^{i_k} \prod_{k=1}^{n} f_k^{j_k} $, with $i_k \leq n-k$, generate $\mathbb{Z}[\tau,\t,f]/I $ as a $\mathbb{Z}[\t]$-module, in particular, they generate $\Z/2[\tau,\t,f]/I$ as a $\Z/2[t]$-module.\\
(I\hspace{-0.5mm}I) Elements $f_1^{{j_1}'}\cdots f_n^{{j_n}'}$ can be written as a linear combination of $f_1^{{j_1}}\cdots f_n^{{j_n}}$ with $j_k \leq 1$ over $\mathbb{Z}/2[\t]$.

\medskip
\noindent
Proof of (I). Clearly the elements $\prod_{k=1}^{n}\tau_k^{i_k} \prod_{k=1}^{n} f_k^{j_k} $, with no restriction on exponents, generate $\mathbb{Z}[\tau,\t,f]/I$ as a $\mathbb{Z}[\t]$-module.
We have an identity 
\begin{eqnarray}\label{ttt}
\prod_{i=1}^{p}{\frac{1}{1-\tau_ix}} &=& \prod_{i=p+1}^{n}(1-\tau_ix)\prod_{i=1}^{n}(1+\tau_ix)\prod_{i=1}^{n}{\frac{1}{1-{\t_i}^2x^2}} \nonumber  \\
 &=& \Big(\sum_{i=0}^{n-p}(-1)^ie_i(\tau_{p+1},{\scriptstyle \cdots},\tau_n)x^i\Big)\Big(\sum_{j=0}^{n}e_j(\tau_1,{\scriptstyle \cdots},\tau_n)x^j\Big)\sum_{k=0}^\infty h_k(t^2)x^{2k} \\
 &=& \Big(\sum_{i=0}^{n-p}(-1)^ie_i(\tau_{p+1},{\scriptstyle \cdots},\tau_n)x^i\Big)\Big(\sum_{j=0}^{n}(2f_j+e_j(\t))x^j\Big)\sum_{k=0}^\infty h_k(t^2)x^{2k}  \nonumber
\end{eqnarray}
where the first equality in \eqref{ttt} follows from \eqref{et2}.

Comparing coefficients of $x^{n+1-p}$ in \eqref{ttt}, we have
\begin{equation} \label{hijk}
h_{n+1-p}(\tau_1,{\scriptstyle \cdots},\tau_p)=\sum_{i+j+2k=n+1-p,\ j+k>0}(-1)^ie_i(\tau_{p+1},{\scriptstyle \cdots},\tau_n)(2f_j+e_j(\t))h_k(\t^2).
\end{equation}

On the other hand, we have
\begin{equation*}
\sum_{j=0}^{i}e_j(\tau_1,{\scriptstyle \cdots},\tau_p)e_{i-j}(\tau_{p+1},{\scriptstyle \cdots},\tau_n)=e_i(\tau)=2f_i+e_i(\t) \text{$\ \ \ \ $for any $i$}, 
\end{equation*}
that is,
\begin{equation} \label{2fiei}
e_{i}(\tau_{p+1},{\scriptstyle \cdots},\tau_n)=2f_i+e_i(\t)-\sum_{j=1}^{i}e_j(\tau_1,{\scriptstyle \cdots},\tau_p)e_{i-j}(\tau_{p+1},{\scriptstyle \cdots},\tau_n)
\text{$\ \ \ \ $for any $i$}.
\end{equation}
Then the same argument as in the latter part of the proof of Lemma~\ref{gene} using \eqref{2fiei} shows that $e_i(\tau_{p+1},{\scriptstyle \cdots},\tau_n)$ can be written as a linear combination of $\prod_{k=1}^{p}\tau_k^{i_k}\prod_{k=1}^{n} f_k^{j_k}$, with $i_k \leq i$, over $\mathbb{Z}[\t]$.
This fact and \eqref{hijk} together with \eqref{comp} show that $\tau_p^{n+1-p}$ is a polynomial in $\tau_1,{\scriptstyle \cdots},\tau_p$, $\t_i$'s and $f_i$'s with the exponent of $\tau_p$ less than or equal to $n-p$. Therefore the elements $\prod_{k=1}^{n}\tau_k^{i_k}\prod_{k=1}^{n} f_k^{j_k}$ with $i_k\leq n-k $, generate $\mathbb{Z}[\tau,\t,f]/I$ as a $\mathbb{Z}[\t]$-module.

\medskip
\noindent
Proof of (I\hspace{-0.5mm}I).  
It follows from Lemma~\ref{lemm:relation} that
$$f_k^2=(-1)^{k+1}\Big(2\sum_{i=1}^{k-1}(-1)^if_if_{2k-i}+\sum_{i=1}^{2k}(-1)^if_ie_{2k-i}(\t) \Big)\quad \text{for $k=1,\dots,n $.}$$ 
In $\mathbb{Z}[\tau,\t,f]/I \otimes \mathbb{Z}/2$, we can disregard $2\sum_{i=1}^{k-1}f_if_{2k-1}$; so ${f_k}^2$ can be written as a linear combination of $f_i$'s over $\mathbb{Z}/2[\t]$. This proves (I\hspace{-0.5mm}I)  and completes the proof of the claim.
\end{proof} 

\begin{lemm}\label{tB2}
$\displaystyle{F(\mathbb{Z}[\tau, \t, f]/I,s)={\frac{1}{(1-s^2)^{2n}}} \prod_{i=1}^n(1-s^{4i})}$.
\end{lemm}
\begin{proof}
The epimorphism \eqref{canoB} means
\begin{equation}\label{aaatukareta}
F(\mHT^*(\mB_n),s) \leq  F(\mathbb{Z}[\tau,\t,f]/I,s).
\end{equation}
In addition, since $\mathbb{Z}[\tau,\t,f]/I$ and $\mHT^*(\mB_n)$ are free over $\mathbb{Z}$, 
\begin{equation}\label{tukareta2}
F(\mHT^*(\mB_n)\otimes \mathbb{Z}/2,s)=F(\mHT^*(\mB_n),s) 
\end{equation}and
\begin{equation}\label{tukareta3}  
F(\mathbb{Z}[\tau,\t,f]/I \otimes \mathbb{Z}/2,s)=F(\mathbb{Z}[\tau,\t,f]/I,s).
\end{equation}
% hold.
It follows from \eqref{moutyotto}, \eqref{aaatukareta}, \eqref{tukareta2} and \eqref{tukareta3} that 
$$F(\mathbb{Z}[\tau,\t,f]/I,s) = F(\mHT^*(\mB_n),s) = {\frac{1}{(1-s^2)^{2n}}} \prod_{i=1}^n(1-s^{4i}), $$
proving the lemma.
\end{proof}
Thus the proof of Theorem~\ref{Btype} has been completed.

%%%%%%%%%%%%%%%%%%%%%%%%%%%%DDDDDDDDDDDDDDDDDDDDDDDDDDDDDDDDDDDDDDDDDDDDDDDDDDDDDDDDD%%%%%%%%%%%%%%%%%%%%%%%%%%%%%%%%%%
\section{Type $D_n$}  \label{sectD}

In this section we will treat type $D_n$.  %The argument in this case is essentially same as type $B_n$ but needs some modification. 
The root system $\Phi(D_n)$ of type $D_n$ is given by 
\begin{equation*} \label{rootDn}
\Phi(D_{n})=\{\pm(\t_i + \t_j), \ \pm(\t_i - \t_j) \mid 1\le i<j\le n \}
\end{equation*}
and its Weyl group is the index two subgroup $\tilde S^+_n$ of $\tilde S_n$ defined by 
\[
\tilde S^+_n:=\{\w\in \tilde S_n\mid \text{the number of $i\in [n]$ with $\w(i)<0$ is even}\}.
\]

\begin{theo} \label{Dtype}
Let $\mD_n$ be the labeled graph associated with the root system $\Phi(D_{n})$ of type $D_n$ above. Then 
\begin{equation} \label{coDn}
%{\small 
%\begin{split}
%&
\mHT^{*}(\mD_n) %\\
=%&
\mathbb{Z}[\tau_1,{\scriptstyle\cdots},\tau_n,\t_1,{\scriptstyle\cdots},\t_n,f_1,{\scriptstyle\cdots},f_{n-1}]/I
,
%\end{split}
%}
\end{equation}
where $I$ is the ideal generated by 
\[
\begin{split}
& 2f_i-e_i(\tau)+e_i(\t ) \quad (i=1,{\scriptstyle\cdots},n-1),\\
&\sum_{j=1}^{2k}(-1)^jf_j(f_{2k-j}+e_{2k-j}(\t)) \quad (k=1,{\scriptstyle\cdots},n), \\
&e_n(\tau)-e_n(\t),
\end{split}
\]
where $f_\ell=0$ for $\ell \ge n$ and $e_\ell(t)=0$ for $\ell >n$.  

\end{theo}

\begin{rema} 
(1) Similarly to $\mD_n$, one can define a labeled graph $\mD_n^-$ with $\tilde S_n\backslash \tilde S_n^+$ as the vertex set on which $\tilde S_n^+$ acts. One sees that $\mHT^*(\mD_n^-)$ agrees with the right hand side of \eqref{coDn} with $e_n(\tau)-e_n(\t)$ replaced by $e_n(\tau)+e_n(\t)$.     

(2) If we set $t_1=\dots=t_n=0$, then the right hand side of the identity in Theorem~\ref{Dtype} reduces to 
\[
\mathbb{Z}[\tau_1,{\scriptstyle\cdots},\tau_n,f_1 ,{\scriptstyle\cdots},f_{n-1}]/J
\end{equation*}
where $J$ is the ideal generated by 
\[
 2f_i-e_i(\tau) \ \ (i=1,{\scriptstyle\cdots},n-1),\quad \sum_{j=1}^{2k-1}(-1)^jf_jf_{2k-j}+f_{2k} \ \ (k=1,{\scriptstyle\cdots},n),\quad e_n(\tau)
\]
where $f_\ell=0$ for $\ell\ge n$, 
and this agrees with the ordinary cohomology ring of the flag manifold of type $D_n$, see \cite[Corollary 2.2]{to-wa74}. 
\end{rema}

\begin{proof}[Outline of proof] 
The proof is almost same as the case of type $B_n$ but needs some modification. We shall list them. 

(1) $e_n(\tau)=e_n(\t)$ in the type $D_n$ case since the number of $i\in [n]$ with $w(i)<0$ is even for $w\in \tilde S^+_n$. 
So $f_n=(e_n(\tau)-e_n(\t))/2=0$ in the case of type $D_n$.

(2) Let $V_i^\pm$ and $\mL_i^\pm$ be defined similarly to the case of type $B_n$.  Then $\mL_i^+$ is naturally isomorphic to $\mD_{n-1}$ but $\mL_i^-$ is not because the number of $j \in [n]\backslash\{ i\}$ with $w(j)<0$ is odd for $w\in \tilde S_n^+$.  Therefore the induction argument as in Lemma~\ref{surj} does not work.  To overcome this, we need to apply the induction argument to $\mD_n$ and $\mD_n^-$ simultaneously because $\mL_i^-$ is isomorphic to $\mD_{n-1}^-$.  Note that if we start with $\mD_n^-$, then $\mL_i^+$ (for $\mD_n^-$) is isomorphic to $\mD_{n-1}^-$ while $\mL_i^-$ (for $\mD_n^-$) is isomorphic to $\mD_{n-1}$.  

(3) If $h\in \mHT^*(\mD_n)$ vanishes on $V_i^+$ for $i<q$, then one can modify $h$ so that it vanishes on $V_i^+$ for $i<q+1$ by subtracting from $h$ a polynomial of the form $G_+^q\prod_{k=1}^{q-1}(\tau_k-\t_n)$ in $\tau_i$'s and $\t_i$'s {with integer coefficients}.  Therefore, we may assume that $h$ vanishes on all $V_i^+$.  Then $h(\w)$ for $\w\in V_1^-$ is divisible by $\prod_{k=2}^n(\t_{\w(k)}-\t_n)=\prod_{k=2}^n(\tau_k-\t_n)(w)$.  (Note that $w$ is connected to a vertex in $V_i^+$ by an edge for $i>1$, but not to any vertex in $V_1^+$. This is the reason why $i=1$ is missing in the product above.)  However, since $f_n=0$ (i.e. $e_n(\tau)=e_n(\t)$) as mentioned in (1) above in the case of type $D_n$, it follows from \eqref{elementB} that 
\begin{equation} \label{V1-}
P:=-\frac{1}{2\t_n}\prod_{k=1}^n(\tau_k-\t_n) = \sum_{k=1}^{n-1}(-1)^{n-1-k}f_i\t_n^{n-1-k}. 
\end{equation}
$P$ is a polynomial in $\t_i$'s and $f_i$'s {with integer coefficients}, vanishes on all $V_i^+$ and takes the value $\prod_{k=2}^n(\t_{\w(k)}-\t_n)$ on $w\in V_1^-$.  Therefore, using the polynomial $P$ in \eqref{V1-}, one can modify $h$ so that it vanishes on all $V_i^+$ and $V_1^-$ by subtracting a polynomial in $\tau_i$'s and $\t_i$'s {with integer coefficients}.  If $h$ vanishes on all $V_i^+$ and $V_j^-$ for $j<q$ with some $q\ge 2$, then one can modify $h$ so that it vanishes on all $V_i^+$ and $V_j^-$ for $j<q+1$ by subtracting from $h$ an integer coefficient polynomial of the form $G_-^qP\prod_{l=1}^{q-1}(\tau_l+\t_n)$.  Therefore we finally reach an element which vanishes on all vertices of $\mD_n$.  This shows that $\mHT^*(\mD_n)$ is generated by $\tau_i$'s, $\t_i$'s and $f_i$'s as a ring. The same argument shows that $\mHT^*(\mD_n^-)$ is also generated by $\tau_i$'s, $\t_i$'s and $f_i$'s as a ring.

(4) A similar argument to the case of type $B_n$ shows that the right hand side in \eqref{coDn} is torsion free over $\Z$ and the Hilbert series of the both sides in \eqref{coDn} coincide, in fact, they are given by ${\frac{1-s^{2n}}{(1-s^2)^{2n}}} \prod_{i=1}^{n-1}(1-s^{4i})$.  The same is true for $\mHT^*(\mD_n^-)$. 
\end{proof}

\bigskip
\noindent
{\bf Acknowledgment.}  The authors would like to thank Takeshi Ikeda, Shizuo Kaji, {Leonardo C. Mihalcea}, Hiroshi Naruse and Takashi Sato   
for information on the integral cohomology rings of flag manifolds and for pointing out mistakes in an earlier version.

\end{document}